\documentclass{amsart}
\usepackage{amsfonts, amsmath, amssymb, amsthm}

\newtheorem{theorem}{Theorem}[section]
\newtheorem{remark}[theorem]{ Remark}

\newtheorem{corollary}[theorem]{Corollary}

\newtheorem{lemma}[theorem]{Lemma}

\newtheorem{definition}[theorem]{Definition}

\newcommand{\beq} {\begin{equation}}
\newcommand{\eeq} {\end{equation}}

\begin{document}
\title{Abundance of isomorphic and non isomorphic intermediate rings }
\author{Bedanta Bose \& Sudip Kumar Acharyya}

\address{Department of Mathematics\\ Swami Niswambalananda Girls' College\\ 115, B.P.M.B. Sarani, Bhadrakali\\ Hooghly - 712232\\ India}
\email{anabedanta@gmail.com}
\address{Department of Pure Mathematics\\ University of Calcutta\\ 35, Ballygunge Circular Road\\ kolkata-700019\\ India}
\email{sdpacharyya@gmail.com}

\begin{abstract}
For a nonpseudocompact space $X$, the family $\sum(X)$ of all intermediate subrings of $C(X)$ which contain $C^*(X)$ contains at least $2^c$ many distinct rings. We show that if in addition $X$ is first countable and realcompact, then there are at least $2^c$ many $C$-type intermediate rings in $\sum(x)$ no two of which are pairwise isomorphic with the special case $X=\mathbb{N}$, it is shown that tere exists a family containing $c$-many pairwise isomorphic $C$-type intermediate rings in $\sum(\mathbb{N})$ 
\end{abstract}
\maketitle
\section{Introduction}
Suppose $C(X)$ is the ring of all real valued continuous functions defined on a Tychonoff topological space $X$. Let $C^*(X)$ be the aggregate of all real valued bounded continuous functions on $X$. Then $C^*(X)$ is a subring as well as sub lattice of $C(X)$. In this article any ring lying between $C^*(X)$ and $C(X)$ will be termed an intermediate ring. Suppose $\sum(X)$ is the collection of all such intermediate rings. There is a common property enjoyed by each ring $A(X)$ in the family $\sum(X)$, viz., that the set of all maximal ideals of $A(X)$ equipped with the hull-kernel topology, often called the structure space of $A(X)$, is topologically equivalent to the Stone-\v{C}ech compactification $\beta X$ of $X$. It is worth mentioning in this context that the last fact was proved first in the \cite{plank} and subsequently in \cite{ls} via a different technique involving local invertibility of functions. 
%In what follows $X$ stands for a completely regular Hausdorff topological space and $C(X)$ is the ring of all real valued continuous functions on $X$. $C^*(X)$ is the subring of $C(X)$ containing all those functions which are bounded over $X$. By an intermediate ring we mean a ring that lies between $C^*(X)$ and $C(X)$. Let $\sum(X)$ be the family of all intermediate rings. It was established first by D. Plank \cite{plank} and then independently by using an entirely different technique by Redlin and Watson \cite{ls} that the structure space of all the intermediate rings are one and the same, viz, the Stone-\v{C}ech compactification $\beta X$ of $X$. The structure space of a commutative ring $R$ is the set of all maximal ideals of $R$ with hull kernel topology.
Intermediate rings have been studied by several authors, viz., \cite{ska},\cite{skd},\cite{hs},\cite{ls},\cite{lsw},\cite{JSack}.   It is well known [see the proof of Theorem 1.2 in \cite{DS}], though we shall give a short proof of the fact that if $X$ is not pseudocompact meaning that $C^*(X)\subsetneq C(X)$, then $\sum(X)$ contains at least $2^{\mathfrak{c}}$ many distinct rings. But it is natural to ask, how many amongst the distinct members of $\sum(x)$ are pairwise non isomorphic as rings. Our main motivation towards writing this article is to address this question. By imposing suitable condition on the topology on $X$, viz., that if $X$ is first countable and realcompact, we realize that there exist a family of $2^{c}$ many distinct intermediate rings  in $C(X)$ each of which is of $C$-type but no two of which are isomorphic [\cite{RC},Theorem 2]. We call  an intermediate ring $A(x) \in \sum (X)$, a $C$-type if $A(X)$ is isomorphic to a ring of the form $C(Y)$ for some Tychonoff space $Y$. We further realize that with the special choice $X=\mathbb{N}$, there exist $2^c$ many pairwise disjoint family of $C$-type intermediate subrings of $C(\mathbb{N}$ with the following property: each family contains $C$-many rings, any two of which are isomorphic (Theorem \ref{Ealpha}). To get access to the above mentioned families of pairwise non isomorphic (respectively pairwise isomorphic) intermediate rings, we have to focus our attention throughout this paper to a special kind of $C$-type intermediate rings. A bit technical elaboration may help us to speak on this point. A subset $Y$ of $X$ is called $C$-embedded in $X$ if each $f\in C(Y)$ has an extension to a $g\in C(X)$. The Hewitt realcompactification $\upsilon X$ of $X$ is the largest subspace of $\beta X$ containing $X$, to which $X$ is $C$-embedded. It is a standard result in the theory of rings of continuous functions that $X$ is pseudocompact meaning $C^*(X)=C(X)$ if and only if $\beta X=\upsilon X$ [\cite{lm}, \S 8A 4]. If $X$ is non pseudocompact, then $\beta X \setminus \upsilon X$ contains a homeomorphic copy of $\beta \mathbb{N}\setminus \mathbb{N}$ and hence $|\beta X\setminus \upsilon X|\geq 2^c$ [\cite{lm}, \S 9D(3)]. With each point $p\in \beta X \setminus \upsilon X$, we associate a $C$-type intermediate ring $C_p$ in such a manner that whenever $p$ and $q$ are distinct points in $\beta X\setminus \upsilon X$, the corresponding rings $C_p$ and $C_q$ are also distinct [Corollary \ref{cor}]. This demonstrates why $\sum (X)$ contains at least $2^c$ many different members with the additional hypothesis of first countability and realcompactness on $X$. It turns out that the rings $C_p$ and $C_q$ are isomorphic when and precisely when the points $p$ and $q$ in $\beta X$ could be exchanged by a homeomorphism on $\beta X$ onto itself [Theorem \ref{C_p}]. This opens the flood gate of non isomorphic intermediate rings of the form $C_p$ [Theorem \ref{conclusion}].

\section{Isomorphism between rings $C_p$}
In this section we establish our principal technical results [Theorem \ref{Ealpha}] and [Theorem \ref{conclusion}] which we deduce as a consequence of [Theorem \ref{C_p}]. Each $f\in C(X)$ has an extension to a uniquely determined continuous map $f^*:\beta X\rightarrow \mathbb{R}^*=\mathbb{R}\cup\{\infty\}$ which is called the Stone extension of $f$ [\cite{lm}, 7.5]. Then $\upsilon X=\{p\in \beta X: f^*(p)\in \mathbb{R}~ \mbox{for each}~ f\in C(X)\}$.
\begin{definition}
For any subset $T$ of $\beta X$, set $C_{T}=\{f\in C(X):f^*(p)\in \mathbb{R}~ \mbox{for each} ~p\in T\}$
\end{definition}

It is easy to prove that $C_{T}$ is an intermediate ring. In particular for any point $p\in \beta X$, $C_{\{p\}}\equiv C_p$ is an intermediate ring.
\begin{definition}
For each subset $A$ of $C(X)$, set $\upsilon_A X=\{p\in \beta X~|~f^*(p)\in \mathbb{R}~\mbox{for each}~f\in A\}$. Then $\upsilon_A X$ is a real compact space lying between $X$ and $\beta X$ [\cite{lm}, \S 8B(3)]. In this terminology $\upsilon_C X=\upsilon X$ and $\upsilon_{C^*} X=\beta X$ 
\end{definition}

It is proved in [\cite{DS}, Theorem 1.5] that for $T\subset \beta X$, $C_T$ is necessarily a $C$-type intermediate ring, in particular therefore for each point $p\in \beta X$, $C_p$ is an intermediate $C$-type ring. The following proposition  is established as Lemma 2.9 in [6].

%In this section we prove our main technical result Theorem \ref{finish} which we derive as a consequence of Theorem \ref{C_p}. A subspace $S$ of $X$ is called $C$-embedded in $X$ if every function of $C(S)$ can be extended to a function of $C(X)$.

%\begin{definition}
%For a point $p$ in $\beta X$ set $C_p=\{f\in C(X)~|~f^*(p)\in \mathbb{R}\}$, here $f^*:\beta X\longrightarrow\mathbb{R}\cup\{\infty\}$ is the continuous extension of $f$.
%\end{definition}
%Since for $f,g \in C_p$ one can easily check that $(f+g)^*(p)=f^*(p)+g^*(p)$ and $(fg)^*(p)=f^*(p)g^*(p)$, it follows that $C_p$ is a subring of $C(X)$. Furthermore as each $f$ in $C^*(X)$ has an extension to a real valued continuous function over the whole of $\beta X$ it is clear that $C^*(X)\subseteq C_p$. Thus $C_p\in \sum(X)$.
%\begin{notation}
%For each subset $A$ of $C(X)$, set $\upsilon_A X=\{p\in \beta X~|~f^*(p)\in \mathbb{R}~\mbox{for each}~f\in A\}$. Then $\upsilon_A X$ is a real compact space lying between $X$ and $\beta X$ [\cite{lm}, \S 8B(3)]. The following proposition is proved as Lemma 2.9 in \cite{Biswajitda} .
%\end{notation}
\begin{lemma}\label{biswajit}
Given any pair of distinct points $p,q$ from $\beta X\setminus \upsilon X$, there exists an $f\in O^p$ such that $f^*(q)=\infty$. Here $O^p=\{g\in C(X)~|~p\in \mbox{int}_{\beta X}\mbox{cl}_{\beta X}Z(g)\}$ and $Z(g)=\{x\in X~|~ g(x)=0\}$ 
\end{lemma}

The next theorem is a consequence of Lemma \ref{biswajit}.
\begin{theorem}\label{a}
If $p\in \beta X\setminus\upsilon X$, then $\upsilon_{C_p}X=\upsilon X\cup\{p\}$.
\end{theorem}
\begin{proof}
It is trivial that $p\in \upsilon_{C_p}X$. But if $q\neq p$ in $\beta X$, then from Lemma \ref{biswajit} there exists an $f\in O^p$ such that $q\notin \upsilon_{f} X$. But $f\in O^p$ implies that $f^*(p)=0$ and hence $f\in C_p$. Thus $q\notin \upsilon_{C_p} X$. It is trivial that $\upsilon X\subseteq \upsilon_{C_p}X$.
\end{proof}
\begin{corollary}\label{cor}
If $p$ and $q$ are distinct points in $\beta X\setminus \upsilon X$, then $C_p$ and $C_q$ are distinct intermediate rings.
\end{corollary}
We are now ready to establish the first important technical result of the paper.
\begin{theorem}\label{C_p}
Let $X$ be a first countable noncompact realcompact space and $p,q\in \beta X\setminus\upsilon X$. Then the ring $C_p$ is isomorphic to the ring $C_q$ if and only if there is a homeomorphism from $\beta X$ onto itself, which permutes $p$ and $q$.
\end{theorem}
\begin{proof}
Each point $x$ in $X$ has a countable local base $\{U_n\}_{n=1}^{\infty}$ in the space $X$. It follows that $\{\mbox{cl}_{\beta X}U_n~|~n\in \mathbb{N}\}$ is a countable local base about the same point in the space $\beta X$ also. Thus $\beta X$ is first countable at each point on $X$. On the other hand no point of $\beta X\setminus X$ is a $G_{\delta}$-point of $\beta X$ [\cite{lm}, \S Corollary 9.6]. Consequently any homeomorphism $\phi:\beta X\longrightarrow \beta X$ permutes points of $X$, i.e., $\phi(X)=X$.

First assume that $C_p$ is isomorphic to $C_q$ under a map $\psi:C_p\longrightarrow C_q$. Then $\psi$ induces a homeomorphism $\Phi$ between their common structure space $\beta X$ in the most obvious manner: $\Phi: \mathfrak{M} \rightarrow \psi(\mathfrak{M})$. As every isomorphism between two rings of real valued continuous functions takes real maximal ideals to real maximal ideals, it follows that the restriction map $\Phi|_{\upsilon_{C_p}X}~:~ \upsilon_{C_p}X\longrightarrow\upsilon_{C_q}X$ becomes a homeomorphism onto $\upsilon_{C_q}X$. Since each homeomorphism on $\beta X$ permutes the points of $X$ as observed above it follows that $\Phi(\upsilon_{C_p}X\setminus X)=\upsilon_{C_q}X\setminus X$. This implies on using Theorem \ref{a} that $\Phi((\upsilon X\setminus X)\cup\{p\})=(\upsilon X\setminus X)\cup\{q\}$. Since $X$ is realcompact, therefore $\upsilon X=X$, and hence $\Phi(p)=q$. 

Conversely let there exist a homeomorphism $\Phi:\beta X\longrightarrow\beta X$ with $\Phi(q)=p$. Choose $f\in C_p$. Then $\Psi(f)=(f^*\circ\Phi)|_X\in C(X)$ and we notice that $(\Psi(f))^*=f^*\circ\Phi$ because these two continuous functions agree on $X$ which is dense in $\beta X$. We further check that $(\Psi(f))^*(q)=(f^*\circ\Phi)(q)=f^*(p)\in\mathbb{R}$. This implies that $\Psi(f)\in C_q$. Thus a map $\Psi~:~C_p\longrightarrow C_q$ is defined in the process. On using the denseness of $X$ in $\beta X$, one can easily verify that $\Psi$ is a ring homomorphism and  $\mbox{Ker}\Psi=\{f\in C_p~:~(f^*\circ\Phi)|_X=0\}=\{f\in C_p~:~f\circ\Phi|_X=0\}=\{f\in C_p~:~f=0\}=\{0\}$. Furthermore for $g\in C_q$, $(g^*\circ\Phi^{-1})(p)=g^{*}(q)$ is a real number. We also note that $g^*\circ\Phi^{-1}=(g\circ\Phi^{-1}|_X)^*$, because of the denseness of $X$ in $\beta X$. Thus $(g^*\circ\Phi^{-1})|_X\in C_p$. We see that $\Psi((g^*\circ\Phi^{-1})|_X)=(g^*\circ\Phi^{-1}\circ\Phi)|_X=g^*|_X=g$. Thus $\Psi:C_p\longrightarrow C_q$ turns out to be an isomorphism.
\end{proof}

The following proposition hints at the possibility of abundance of nonisomorphic intermediate rings.

\begin{theorem}\label{flood}
Let $X$ be first countable noncompact and realcompact. If $p$ is a $P$-point of $\beta X\setminus X$ and $q$ a non $P$-point of $\beta X\setminus X$, then $C_p$ is not isomorphic to $C_q$. ($P$-points of a space $X$ are those points whose neighbourhood filters are closed under countable intersections).
\end{theorem}
\begin{proof}
Follows immediate from Theorem \ref{C_p}, because the property that a point in a topological space is a $P$-point is a topological invariant.
\end{proof}
\begin{remark}\label{iso cp}
The above Theorem \ref{flood} tells that if $C_p$ is isomorphic to $C_q$ and $p$ is a P-point then $q$ is also a P-point. But it is not known to us whether the hypothesis that $p$ and $q$ are P-points leads to the conclusion that $C_p$ and $C_q$ are isomorphic.
\end{remark}

To set further insight into the possibility of existense of nonisomorphic intermediate rings we need a classification of points of $\beta \mathbb{N}\setminus \mathbb{N}$. By using the extension property of Stone-\v{C}ech compactification it follows that every permutation $\tau :\mathbb{N}\rightarrow \mathbb{N}$ has an extension to an automorphism $\tau^*:\beta \mathbb{N}\rightarrow\beta\mathbb{N}$. On the other hand since each point of $\mathbb{N}$ is isolated in $\beta\mathbb{N}$, while no point of $\beta \mathbb{N}\setminus\mathbb{N}$ is isolated in $\beta\mathbb{N}$, we get that if $\psi:\beta\mathbb{N}\rightarrow \beta \mathbb{N}$ is a homeomorphism, then $\Psi(\mathbb{N})=\mathbb{N}$. This means that $\Psi|_{\mathbb{N}}$ is a permutation of $\mathbb{N}$ and $(\Psi|_{\mathbb{N}})^*=\Psi$. thus the entire collection of automorphism on $\beta \mathbb{N}$ is given by the family $\{\tau^* : \tau ~\mbox{is a permutation of}~\mathbb{N}\}\equiv \{\tau^* : \tau \in Per(\mathbb{N})\}$, writing $Per(\mathbb{N})$ for the set of all permutation on $\mathbb{N}$.

\begin{definition} \label{equivalent}
An equivalence relation $"\rightarrow"$ on $\beta \mathbb{N}\setminus \mathbb{N}$ is defined as follows : for $p,q\in \beta\mathbb{N}\setminus \mathbb{N}$, $p\rightarrow p$ if there exists $\tau \in Per(\mathbb{N})$ such that $q=\tau^*(p)$.
\end{definition}
Let $\{E_{\alpha}:\alpha\in \Lambda\}$ be the set of the corresponding disjoint classes, thus $(\beta\mathbb{N}\setminus \mathbb{N})/\rightarrow =\{E_\alpha : \alpha\in \Lambda\}$.

Before proceeding further, we recall from  \cite{RC}, the notion orbit of a point  in $\beta \mathbb{N}\setminus \mathbb{N}$. For any point $p\in \beta \mathbb{N}\setminus \mathbb{N}$, the set $Orbit(p)=\{\delta(p):\delta\in Aut(\beta \mathbb{N}\}$ is called the orbit of $p$, here $Aut(\beta \mathbb{N})$, stands for the set of all homeomorphism from $\beta \mathbb{N}$ onto itself.

We have observed above that an automorphism on $\beta \mathbb{N}$ sends point of $\beta \mathbb{N}\setminus \mathbb{N}$ to points in $\beta \mathbb{N}\setminus \mathbb{N}$. Furthermore there are just $c$-many permutation on $\mathbb{N}$ and consequently there are precisely $c$-many automorphism on $\beta \mathbb{N}$. Hence for any point $p\in \beta \mathbb{N}\setminus \mathbb{N}$, $Orbit(p)\subseteq \beta \mathbb{N}\setminus \mathbb{N}$ and $|Orbit(p)|\leq c$. But there exists a family $\{A_{\alpha}\}_{\alpha}$ of $c$-many pairwise disjoint nonempty open sets in the space $\beta \mathbb{N}\setminus \mathbb{N}$ [\cite{lm}, \S 6Q]. Since $Orbit(p)$ is a dense subset of $\beta \mathbb{N}\setminus \mathbb{N}$, a fact reproduced in [\cite{RC}, Corollary 3.20] it therefore follows that $|Orbit(p)|\geq c$. Thus $|Orbit(p)|=c$. for any point $p$ in $\beta \mathbb{N}\setminus \mathbb{N}$. On recalling the quotient set $(\beta\mathbb{N}\setminus \mathbb{N})/\rightarrow =\{E_\alpha : \alpha\in \Lambda\}$ introduced after Definition \ref{equivalent}, the following result is therefore evident:
\begin{remark}
For $\alpha \in \Lambda$ and any point $p$ on $E_{\alpha}$, $Orbit(p)=E_{\alpha}$ and therefore $|E_{\alpha}|=c$. Consequently therefore, since $|\beta\mathbb{N}\setminus \mathbb{N}|=2^c$, the cardinal number of the index set $\Lambda$ is $2^c$, this means that there are $2^c$ many different orbits of points of $ \beta\mathbb{N}\setminus \mathbb{N}$. the following theorem manifests that there do exist uncountably many isomorphic $c$-type intermediate subrings of $C(\mathbb{N})$
\end{remark}
\begin{theorem} \label{Ealpha}
Let $E_{\alpha}, \alpha \in \Lambda$ be chosen arbitrarily and choose any pair of distinct points $p,q$ from $E_{\alpha}$. Then $C_p$ is isomorphic to $C_q$.
\end{theorem}
\begin{proof}
Since $p,q\in E_{\alpha}$, it follows that $p\rightarrow q$, this means that there is a homeomorphism $\psi:\beta \mathbb{N}\rightarrow \beta\mathbb{N}$ with $\psi(p)=q$. We now apply Theorem \ref{a} to infer that $C_p$ is isomorphic to $C_q$. 
\end{proof}

Now Construct a set $S$ by picking up exactly one member $p_{\alpha}$ from each set $E_{\alpha}$, for $\alpha \in \Lambda$. We write $S=\{p_{\alpha}|\alpha \in \Lambda\}$. Then $S$ contains $2^c$ many elements. We are now ready to enunciate the last main proposition of this paper.

\begin{theorem}\label{conclusion}
Suppose $X$ is a first countable real compact space which is not pseudocompact. Then there exists a family containing no fewer than $2^c$ many members of pairwise nonisomorphic $C$-type intermediate rings lying between $C^*(X)$ and $C(X)$. 
\end{theorem}
\begin{proof}
Since $X$ is not pseudocompact there exists a homeomorphic copy of $\mathbb{N}$, which is already $C$-embedded in $X$ [\cite{lm}, 1.21]. It is clear that $cl_{\beta X}\mathbb{N}$ is a Hausdorff compactification of $\mathbb{N}$ in which $\mathbb{N}$ is $C^*$-embedded. Therefore we can write $cl_{\beta X}\mathbb{N} = \beta \mathbb{N}$. But $X$ is $C$-embedded in $\upsilon X$ implies that $\mathbb{N}$ is $C$-embedded in $\upsilon X$. As $C$-embedded countable subsets of a completely regular space are closed subsets of it [\cite{lm}, \S 3B3], it follows that $\mathbb{N}$ is closed in $\upsilon X$ and consequently $\mathbb{N}$ is closed in $X$ become $X$ is realcompact and furthermore $\beta\mathbb{N}\setminus\mathbb{N}\subseteq \beta X\setminus X$. Now $p_{\alpha}$ and $p_{\beta}$ are two distinct points chosen from the set $S$, defined above,just preceding the statement of the present theorem, then no homeomorphism on $\beta \mathbb{N}$ onto itself can permute $p_{\alpha}$ and $p_{\beta}$. We claim that there does not exist any homeomorphism from $\beta X$ onto itself, which can permute $p_{\alpha}$ and $p_{\beta}$ and hence by Theorem \ref{C_p}, the ring $C_{p_{\alpha}}=\{f\in C(X): f^*(p_{\alpha})\in \mathbb{R}\}$ is not isomorphic to the ring $C_{p_{\beta}}=\{f\in C(X): f^*(p_{\beta})\in \mathbb{R}\}$ and that finishes our theorem. We argue by contradiction and assume that there is a homeomorphism $\psi:\beta X \rightarrow \beta X$ such that $\psi(p_{\alpha})=p_{\beta}$. Since $X$ is first countable, We have observed in the proof of Theorem \ref{C_p} that $\psi|_{X} :X\rightarrow X$ is a homeomorphism. As $\mathbb{N}$ is $C$-embedded in $X$, it follows that $\psi(\mathbb{N})$, a copy of $\mathbb{N}$ is $c$-embedded in $X$. On the otherhand the homeomorphic nature of $\psi$ on $\beta X$ implies that $\psi(cl_{\beta X}\mathbb{N})=cl_{\beta X}\psi(\mathbb{N})$, in otherwords, abusing notations, we can write $\psi(\beta \mathbb{N})=\beta\mathbb{N}$ and consequently $\psi(\beta\mathbb{N}\setminus\mathbb{N})=\beta\mathbb{N}\setminus\mathbb{N}$. Thus $\psi|_{\beta\mathbb{N}\setminus\mathbb{N}}:\beta\mathbb{N}\setminus\mathbb{N}\rightarrow \beta\mathbb{N}\setminus\mathbb{N}$ is a homeomorphism, which by our initioal assumption made above exchanges $p_{\alpha}$ and $p_{\beta}$. Hence $p_{\alpha}\rightarrow p_{\beta}$ (vide the equivalence relation Definition \ref{equivalent}) and therefore $p_{\alpha}=p_{\beta}$, because of the nature of the set $S$ constructed preceeding this theorem. This is a contradiction.
\end{proof}
\begin{remark}
In the proof of Theorem \ref{C_p}, the first countability on $X$ is exploited only to prove that any homeomorphism on $\beta X$ onto $\beta X$ exchanges the points of $X$. Therefore if the hypothesis of Theorem \ref{C_p} is weakened in the manner that, $X$ is a noncompact realcompact space such that any homeomorphism on $\beta X$ onto $\beta X$ permutes the points of $X$, then the conclusion of Theorem \ref{C_p} and thereafter the conclusion of Theorem \ref{flood}, Remark \ref{iso cp} and Theorem \ref{Ealpha} are still valid. Indeed it is not difficult to produce an example of a space $X$ with the above mentioned weekend version of the hypothesis of Theorem \ref{C_p}, by dropong the first countability condition. For that purpose let $X$ be a countable non compact Tychonoff space such that each point of $X$ is the limit of a converging sequence of distinct terms. Then $X\varsubsetneq \beta X$ and for any point $p\in \beta X\setminus X$, there does not exist any convergent sequence $\{x_{1},x_{2},...\}$ of distinct terms in $\beta X$ which can converge to $p$. We establish this claim by contradiction and  assume that there exist such a convergent sequence $\{x_{1},x_{2},...\}$ in $\beta X$ converging to $p$ with $p\neq x_n$ and for each $n\in \mathbb{N}$ $D$ is the range of this sequence, then $D\cup X$ becomes a regular Lindeloff subspace of $\beta X$ and therefore normal and $D$ becomes a closed discrete subspace of $D\cup X$. By Tietze's extension Theorem $D$ is $C^*$-embedded in $D\cup X$ and consequently $D$ becomes $C^*$-embedded in $\beta X$. Now the subsets $\{x_1,x_3,x_5,....\}$ and $\{x_2,x_4,x_6,...\}$ of $D$ are completely separated by a continuous function in $C(D)$, because of discreteness of $D$, but the closure of these two sets in $\beta X$ are not disjoint because they meet at $p$. By applying Urysohn's extension theorem [cite, Teorem 1, 17], we arrive at a contradiction. Thus no point in $\beta X\setminus X$ is the limit of a converging sequence of distinct terms in $\beta X$. Since by our initial assumption every point of $X$ is the limit of a sequence of distinct terms, it follows for any homeomorphism $T:\beta X\rightarrow \beta X$ onto $\beta X$, $T(X)=X$. Thus it is proved that any countable noncompact space $X$ even if it is not first countable but, each point of which is the limit of a sequence of distinct terms can be a candidate for the space $X$ in the statement of Theorem \ref{C_p}.
\end{remark}

\textbf{Acknowledgement:} The authors are grateful to Professor Alan Dow, who is kind enough to provide us the necessary example in Remark 2.13.

\end{document}